\documentclass[12pt]{amsart}

\usepackage{amssymb}
\usepackage{amsmath}
\usepackage{color}

\begin{document}
\numberwithin{equation}{section} 

\def\1#1{\overline{#1}}
\def\2#1{\widetilde{#1}}
\def\3#1{\widehat{#1}}
\def\4#1{\mathbb{#1}}
\def\5#1{\mathfrak{#1}}
\def\6#1{{\mathcal{#1}}}

\def\C{{\4C}}
\def\R{{\4R}}
\def\N{{\4N}}
\def\Z{{\4Z}}
\def\T{{\Theta}}

\def\sideremark#1{\ifvmode\leavevmode\fi\vadjust{
\vbox to0pt{\hbox to 0pt{\hskip\hsize\hskip1em
\vbox{\hsize1cm\tiny\raggedright\pretolerance10000
\noindent #1\hfill}\hss}\vbox to8pt{\vfil}\vss}}}
\def\sr#1{\sideremark{#1}}

\def\highlight#1{\color{red} #1\color{black}}
\def\hl#1{\highlight{#1}}

\title[]{Formally-Reversible Maps of $\C^2$}
\author[A. O'Farrell \& D. Zaitsev]{Anthony G. O'Farrell \& Dmitri Zaitsev*}
\address{%
A. O'Farrell: Mathematics Department,
 NUI, Maynooth, Co. Kildare, Ireland}
 \email{admin@maths.nuim.ie}
 \address{
D. Zaitsev: School of Mathematics, 
Trinity College Dublin, Dublin 2, Ireland
}
\email{zaitsev@maths.tcd.ie}
\thanks{$^{*}$Supported in part by the Science Foundation Ireland grant 10/RFP/MTH2878.}


\def\Label#1{\label{#1}}


\def\cn{{\C^n}}
\def\cnn{{\C^{n'}}}
\def\ocn{\2{\C^n}}
\def\ocnn{\2{\C^{n'}}}


\def\dist{{\rm dist}}
\def\const{{\rm const}}
\def\rk{{\rm rank\,}}
\def\id{{\sf id}}
\def\aut{{\sf aut}}
\def\Aut{{\sf Aut}}
\def\CR{{\rm CR}}
\def\GL{{\sf GL}}
\def\SL{{\sf SL}}
\def\U{{\sf U}}
\def\Re{{\sf Re}\,}
\def\Im{{\sf Im}\,}
\def\im{{\rm im}\,}
\def\span{\text{\rm span}}
\def\tr{{\sf tr}\,}
\def\HOT{{\textup{\small HOT}}}

\def\codim{{\rm codim}}
\def\crd{\dim_{{\rm CR}}}
\def\crc{{\rm codim_{CR}}}

\def\phi{\varphi}
\def\eps{\varepsilon}
\def\d{\partial}
\def\a{\alpha}
\def\b{\beta}
\def\g{\gamma}
\def\G{\Gamma}
\def\D{\Delta}
\def\Om{\Omega}
\def\k{\kappa}
\def\l{\lambda}
\def\L{\Lambda}
\def\z{{\bar z}}
\def\w{{\bar w}}
\def\Z{{\mathbb Z}}
\def\t{\tau}
\def\th{\theta}

\emergencystretch15pt
\frenchspacing

\newtheorem{Thm}{Theorem}[section]
\newtheorem{Cor}[Thm]{Corollary}
\newtheorem{Pro}[Thm]{Proposition}
\newtheorem{Lem}[Thm]{Lemma}

\theoremstyle{definition}\newtheorem{Def}[Thm]{Definition}

\theoremstyle{remark}
\newtheorem{Rem}[Thm]{Remark}
\newtheorem{Exa}[Thm]{Example}
\newtheorem{Exs}[Thm]{Examples}

\def\bl{\begin{Lem}}
\def\el{\end{Lem}}
\def\bp{\begin{Pro}}
\def\ep{\end{Pro}}
\def\bt{\begin{Thm}}
\def\et{\end{Thm}}
\def\bc{\begin{Cor}}
\def\ec{\end{Cor}}
\def\bd{\begin{Def}}
\def\ed{\end{Def}}
\def\br{\begin{Rem}}
\def\er{\end{Rem}}
\def\be{\begin{Exa}}
\def\ee{\end{Exa}}
\def\bpf{\begin{proof}}
\def\epf{\end{proof}}
\def\ben{\begin{enumerate}}
\def\een{\end{enumerate}}
\def\beq{\begin{equation}}
\def\eeq{\end{equation}}

\def\dsty{\displaystyle }

\begin{abstract}
An element $g$ of a group is called {\em reversible}
if it is conjugate in the group to its inverse.  
This paper is about reversibles in the group
$\5G=\5G_2$ of formally-invertible pairs of 
formal power series in two 
variables, with complex coefficients.
The main result is a description
of the generic reversible elements of $\5G_2$. 
We list two explicit sequences of reversibles
which between them represent all the
conjugacy classes of such reversibles.  We show that
each such element is reversible by some element of finite
order, and hence is the product of two elements of finite even 
order.  Those elements that may be reversed by an involution are
called {\em strongly reversible}. We also characterise these.

We draw some conclusions about
generic reversibles in the group $\6G=\6G_2$ of biholomorphic
germs in two variables, and about the factorization of formal maps
as products of reversibles. Specifically,
each product of reversibles reduces to the product of five.
\end{abstract}

\maketitle

\section{Introduction}
An element $g$ of a group is called {\em reversible}
if it is conjugate to its inverse, i.e. the conjugate $g^h:=h^{-1}gh$ equals $g^{-1}$ for some $h$ from the group.  We say that $h$ {\em reverses} $g$ or $h$ {\em is a  reverser of} $g$, in this case.
Furthermore, if the reverser $h$ can be chosen to be an involution, $g$ is called {\em strongly reversible}.
(Note that some literature uses the terminology ``weakly reversible'' and ``reversible''
instead of respectively ``reversible'' and ``strongly reversible'' used here.)

Reversible maps have their origin in problems of classical dynamics,
such as the harmonic oscilator, the $n$-body problem or the billiards
and Birkhoff \cite{B} was one of the first to realize their significance.
He observed that a Hamiltonian system with Hamiltonian quadratic in 
the momentum $p$ (such as the $n$-body problem) or, more generally
any system in  of the form 
\begin{equation}\Label{system}
\begin{cases}
{\dsty\d q}/{\dsty\d t} = Lp,\\
{\dsty\d p}/{\dsty\d t} = V(p,q),
\end{cases}
(q,p)\in\R^{n}\times\R^{n}, \quad t\in\R,
\end{equation}
where $L$ is linear, admits the so-called
``time reversal symmetry'' $(t,q,p)\mapsto (-t,q,-p)$.
In particular, the flow map $(q_{0},p_{0})\mapsto (q(t),p(t))$,
where $(q(t),p(t))$ is the solution of \eqref{system}
with $(q(0),p(0))=(q_{0},p_{0})$ is reversed by the involution 
$(q,p)\mapsto(q,-p)$.

In CR geometry reversible maps played important role in the celebrated 
work of Moser and Webster \cite{MW},
arising as products of two involutions
naturally associated to a CR singularity.
Such a reversible map is called there ``a discrete version of the Levi form"
and plays a fundamental role
in the proof of the convergence of the normal form
for a CR singularity.
More recently, this map has been used
by Ahern and Gong \cite{AG}
for so-called parabolic CR singularities.

There is a class of diffeomorphisms of $\R^2$
that has received considerable attention: the so-called 
{\em standard maps}, or {\em Taylor-Chirikov maps},
which arise in many applications in physics, and
are strongly-reversible in the group of real-analytic
diffeomorphisms \cite[Section 1.1.4]{MK}.  
The standard maps fix the origin and
are area-preserving. When they are real-analytic, they
thus give reversible elements of our power-series group.

\medskip
This paper aims to classify formally reversible maps in two complex variables,
for which the eigenvalues of the linear part are not roots of unity.
The presence of roots of unity is a well-known obstruction
leading to the presence of additional resonances.

Reversibility is already
understood \cite{O} in the group $\5G_1$
of formal power series maps of $\C$ without constant terms
(see also \cite{AO} for reversibility in the group $\6G_1$ of
biholomorphic map germs of $\C$ fixing the origin):  

\bt\cite[Theorem 5]{O}
A formal power series self-map of $\C$ without constant term is reversible if and only if it is formally conjugate to 
\begin{equation}\Label{phi}
\phi_{\mu,\l,k}(z):=\frac{\mu z}{(1+\l z^{k})^{1/k}}
\end{equation}
for some integer $k\ge 1$, $\mu=\pm1$ and $\l\in\{0,1\}$.
The map \eqref{phi} is reversed by any rotation $z\mapsto \omega z$
with $\omega^{k}=-1$.
\et
Note that for $k=1$, $\phi_{1,1,1}$
is precisely the (unique up to conjugation) map $z\mapsto z+\HOT$
that is conjugate to a projective transformation,
whereas $\phi_{1,1,k}$ is obtained from $\phi_{1,1,1}$
via ``conjugation'' under the non-invertible map $z\mapsto z^{k}$,
i.e.\ 
$(\phi_{1,1,k}(z))^{k}=\phi_{1,1,1}(z^{k})$.

In this paper, our main focus is on the group
$\5G=\5G_2$ of formal power series self-maps of $\C^{2}$
without constant terms. 
We obtain the following formal normal form:

\bt\Label{main0}
A formal power series map of $\C^{2}$ whose linear part
has eigenvalue $\l$ not a root of unity is reversible if and only if it is 
(formally) conjugate either to its linear part
 or to one of the following maps, which are all
pairwise inequivalent under conjugation:
\beq\Label{E:normal-0'} 
(z_1,z_2) \mapsto \left( \l(1+p^k)z_{1}, \frac{1}{\lambda(1+p^k)} z_{2}
\right), \quad
 \left(
\frac{\lambda (1+ p^k)^{\frac1k}}{(1+2p^{k})^{\frac1k}} z_{1},
\frac{1}{\lambda(1+ p^k)^{\frac1k}} z_{2}
\right), \quad p=z_{1}z_{2},
\eeq
where $k\ge1$ is an integer. Furthermore, the following hold:
\begin{enumerate}
\item 
Each map in \eqref{E:normal-0'}
is reversed by any $J_{c}(z_{1},z_{2})=c(z_{2},z_{1})$
such that $c^{2k}=1$ for the first series of maps and $c^{2k}=-1$ for the second.
\item A map in \eqref{E:normal-0'} (with $\l$ not a root of unity) is strongly reversible if and only if 
it is in the first series.
\item Each reverser of a map in \eqref{E:normal-0'} has finite order.
\end{enumerate}
\et
Similarly to \eqref{phi}, the maps \eqref{E:normal-0'}
can be obtained from those corresponding to $k=1$
by a ``conjugation'' under the non-invertible map
$(z_{1},z_{2})\mapsto (z_{1}^{k},z_{2}^{k})$.
The \lq generic type' condition that $\l$ 
not be a root of unity
is important as the following example shows.
(Note that a similar condition also appears in the work
of Moser-Webster \cite{MW}, where it is also shown to be crucial
for the existence of the formal normal form constructed there.)

\be\Label{exa}
The assumption that $\l$ is not root of unity cannot be dropped in Theorem~\ref{main0}. Indeed, the map
\begin{equation}
F(z_{1},z_{2})=\left(
\frac{z_{1}}{1+z_{1}},
\frac{z_{2}}{1+z_{1}}
\right)
\end{equation}
is reversed by the involution
$(z_{1},z_{2})\mapsto(-z_{1},z_{2})$,
(in fact, it is the projectivization of the linear map
$(z_{0},z_{1},z_{2})\mapsto (z_{0}+z_{1},z_{1},z_{2})$
reversed by $(z_{0},z_{1},z_{2})\mapsto (z_{0},-z_{1},z_{2})$).
However, it 
has the second order terms
$(-z_{1}^{2},-z_{2}z_{1})$
that cannot be eliminated by conjugation
and do not occur in a map from \eqref{E:normal-0'}.
Hence, $F$ is not conjugate to any map in \eqref{E:normal-0'}.
\ee

We also obtain polynomial representatives conjugate to the 
maps of the second series:
\bp\Label{polynomial}
For every $k$, the map from the second series in \eqref{E:normal-0'}
is formally conjugate to the polynomial map
$$ 
(z_{1}, z_{2})\mapsto
\left(
\lambda (1+ p^k)z_1,
{\lambda}^{-1}\left(1+ p^k+(2k+1)p^{2k}z_{2}\right)
\right), \quad p=z_{1}z_{2}.
$$
\ep

\br
Moser \cite{M1956} showed that real-analytic area preserving
maps germs on  $(\R^2,0)$ having a hyperbolic fixed point at the origin
may be conjugated by a convergent coordinate
change to the normal form  
\begin{equation}\Label{area}
\left(\sigma\cdot x\cdot\exp(w(xy)),
\sigma\cdot y\cdot\exp(-w(xy))\right),
\end{equation}
where
$\sigma=\pm1$ and $w$ is a convergent series. Thus they
are strongly reversible and are conjugate to
an element of our first series in \eqref{E:normal-0'}, with some real $\l\not=\pm1$.
On the other hand, a general map conjugate to \eqref{area} is always strongly reversible but need not be area preserving, whereas maps of the second series in \eqref{E:normal-0'} are not area-preserving.  
\er

Theorem~\ref{main0} and Proposition~\ref{polynomial} are proved in Section~\ref{classes}.

Furthermore, following the arguments of Moser-Webster \cite{MW} we come to
a biholomorphic classification for the maps conjugate to the first series:

\bt\Label{convergence}
If a biholomorphic map germ in $\C^{2}$ is formally conjugate to a map in the first series in \eqref{E:normal-0'} with $|\l|\ne 1$, then 
it is biholomorphically conjugate to it.
\et

Note that any map in the first series in \eqref{E:normal-0'}
is conjugated to itself by (i.e. commutes with)
any map of the form
$(z_{1},z_{2})\mapsto \big(z_{1}\phi(p),{z_2}/{\phi(p)}
\big)$,
where $\phi$ is a formal power series in $t\in\C$ with $\phi(0)\not=0$.
Hence the biholomorphic conjugation map may differ
from the original formal one.

The statement of Theorem~\ref{convergence} does not necessarily hold
for
maps conjugate to the 
second second series in \eqref{E:normal-0'}.
In fact, Example~\ref{divergent} below
shows that a biholomorphic map
conjugate to a map in the second series in \eqref{E:normal-0'}
(and hence formally reversible),
may fail to 
be biholomorphically conjugate to its formal normal form
and may even fail to be 
 biholomorphically reversible.

The statement of Theorem~\ref{convergence} also does not hold
without the assumption $|\l|\ne 1$
due to the remarkable theorem of Gong \cite[Theorem~1.1]{G}
stating the existence of a biholomorphic map
$F(z_{1},z_{2}) =(\l z_{1},\l^{-1}z_{2}) + \HOT$
with $\l$ not a root of unity and $|\l|=1$
such that 
$F$ is formally reversible by an involution
(and hence $F$ is formally conjugate to a map in the first series in \eqref{E:normal-0'})
but is not reversible by a biholomorphic involution
(and hence $F$ is not biholomorphically conjugate to its formal normal form).


We shall further discuss reversibility and centralisers in some related
groups, along the way.  We draw some conclusions about
generic reversibles in the group $\6G_2$ of biholomorphic
germs in two variables (Theorem~\ref{biholo}), and about the factorization of maps
as products of reversibles.
Since a reversible element of $\5G_{2}$
has linear part of determinant $\pm1$ (see \S\ref{linear}),
any product of them has the same property.
Conversely we show:

\bt\Label{prod}
Each element of $\5G_{2}$  whose linear part has determinant
$\pm1$ is the product of at most four reversibles and an involution.
\et

Theorem~\ref{prod} is proved in \S\ref{factor}.

\section{Notation and Preliminaries}
\subsection{The map $L$}\label{l-notation}
A typical element $F\in\5G$ takes the form
$$F(z)=(F_{,1}(z),F_{,2}(z))=(F_{,1}(z_1,z_2),F_{,2}(z_1,z_2))$$
where each $F_{,j}(z)$ is a power series in two variables
having complex coefficients, and no constant term.
We shall refer to such series $F$ as maps, even though
they may be just \lq formal', i.e. 
the series may fail to converge at any $z\not=0$.

We usually write the formal composition of two maps
$F,G\in\5G$ as FG. We also write the product of two complex
numbers $a$ and $b$ as $ab$, but in cases where there might
be some ambiguity we use $a\cdot b$.
 
The series $F$ may be expressed
as a sum  
$$F = \sum_{k=1}^\infty L_k(F), $$
where $L_k(F)$ is homogeneous of degree $k$. 
We abbreviate $L_1(F)$ to $L(F)$. This term,
the {\em linear part of} $F$, belongs to
the group
$\GL=\GL(2,\C)$.
We have inclusions
$\GL \to \5G$,
and $L:\5G\to \GL$ is a group homomorphism.

The elements of the kernel of $L$ 
are said to be {\em tangent to the identity}.

\subsection{Elements of Finite Order}
We note the following, in which $\5G_n$ denotes
the group of all invertible formal power series self-maps of $\C^{n}$
without constant coefficients.  (Various cases of this lemma,
and the idea of its simple proof,
are well-known.  The case $n=1$ is
very classical. For holomorphic maps
in $n$ variables, see \cite[p.298]{AR}.
In the differentiable category, 
Montgomery and Zippin \cite{MZ} 
proved the  local conjugacy of any involution to
its derivative at a fixed point.  
The equivalent global question is more 
delicate.)
\begin{Lem}\Label{L:fol}
Let $n\in\N$ and let $\5H$ be a subgroup of $\5G_n$
such that\\
(1) $L(F)\in \5H$ whenever $F\in\5H$, and
\\
(2) $\5H\cap\ker L$ is closed under
convex combinations, i.e.
if $F_1,F_2\in\5H$, $L(F_1)=L(F_2)=\id$
and $0<\alpha<1$, then $\alpha F_1+(1-\alpha)F_2\in\5H$.
\\
Suppose $\T\in\5H$ has finite order. Then $\T$ is conjugated
by an element of $\5H\cap \ker L$ to its linear part $L(\T)$.
\end{Lem} 
\begin{proof} Suppose $\T^k=\id$. Let $T=L(\T)$, and form
\begin{equation}\Label{inverse}
 H = \frac1{k}\left(
	\id + T^{-1}\T+T^{-2}\T^2+\cdots+T^{-(k-1)}\T^{k-1}
\right).
\end{equation}
Then the assumptions imply that $H\in\5H$,
and a short calculation shows that $T^{-1}H\T = H$,
so that $T^H=\T$, as required.
\end{proof}

This applies to $\5H=\5G$, $\5G\cap\ker L$,
 $\5G\cap\ker(\det\circ L)=L^{-1}(\SL(2,\C))$,
$L^{-1}(\U(2,\C))$ (and, more generally to $L^{-1}(H)$
for any subgroup $H\le\GL$), to the corresponding
subgroups of biholomorphic germs (i.e. series that converge
on a neighbourhood of the origin) and to other subgroups introduced
below.  It applies to the intersection of any two groups
to which it applies.

In particular, in any $\5H$ to which the lemma applies,
each involution is conjugate to one of the linear involutions
in the group.  In $\GL=\GL(n,\C)$, a matrix is an involution
if and only if it is diagonalizable with eigenvalues $\pm1$.

\subsection{Reversibles in one variable}
Here we collect facts about reversibles in one variable
that will be used throughout the paper.

\bl
A map $h\in \5G_{1}$ is a reverser of the map \eqref{phi} with $\mu=1$ and $\l\ne0$
if and only if it is of the form
\begin{equation}
h_{\omega,\nu}(z)=\frac{\omega z}{(1+\nu z^{k})^{1/k}},
\end{equation}
where $\nu,\omega\in\C$ are arbitrary with $\omega^{k}=-1$.
\el
\bpf
Since $\phi_{1,\l,k}$ is the inverse of $\phi_{1,-\l,k}$,
the map $h_{\omega,0}(z)=\omega z$ reverses $\phi_{1,\l,k}$ for any $\omega$ with $\omega^{k}=-1$.
Furthermore, since any map $h_{1,\nu}$ commutes with $\phi_{1,\l,k}$, we have
$$h_{\omega,\nu}^{-1} \phi_{1,\l,k} h_{\omega,\nu} =  
h_{1,\nu}^{-1} h_{\omega,0}^{-1} \phi_{1,\l,k} h_{\omega,0} h_{1,\nu} 
=h_{1,\nu}^{-1} \phi_{1,\l,k}^{-1} h_{1,\nu} = \phi_{1,\l,k}^{-1}$$
and therefore any $h_{\omega,\nu}$ reverses $\phi_{1,\l,k}$.
Vice versa, if $h$ reverses $\phi_{1,\l,k}$, comparing the coefficients
of $z^{k+1}$ in the equation 
\begin{equation}\label{hphi}
h\phi_{1,-\l,k} = \phi_{1,\l,k}h
\end{equation}
yields $h(z)=\omega z+\textup{O}(z^2)$ for some $\omega^{k}=-1$.
Furthermore, we can choose $\nu$ such that
$g:=h_{\omega,\nu}^{-1} h$ has no coefficient of $z^{p+1}$.
Then it follows from \eqref{hphi} that $g$ commutes with $\phi_{1,-\l}$.
We claim that $g=\id$. Suppose on the contrary, that $g(z)=z+az^{\b}+\cdots$
with $a\ne 0$ and $\b\ne p+1$.
Then comparing the coefficients of $z^{\beta+k+1}$ 
in the identity $g\phi_{1,\l,k}=\phi_{1,\l,k}g$ yields a contradiction.
Hence $g=\id$ proving that $h$ is of the form \eqref{hphi}.
\epf

\bc\Label{1-finite}
Any reverser of the map \eqref{phi} with $\mu=1$ and $\l\ne0$
is of finite order at most $2k$. 
\ec
\bpf
Since $h_{\omega,0}$ reverses $\phi_{1,\l,k}$, we have
$$h_{\omega,\nu}^{2k}=(h_{\omega,0}\phi_{1,\nu,k} h_{\omega,0}\phi_{1,\nu,k})^{k}
=(h_{\omega,0} h_{\omega,0}   \phi_{1,\nu,k}^{-1} \phi_{1,\nu,k})^{k}
=h_{\omega,0}^{2k}=\id.$$
\epf

\subsection{Linear reversibles}\Label{linear}
Reversibility is preserved by homomorphisms,
so a map $F\in\5G_{n}$ is reversible only if 
$L(F)$ is reversible in $\GL(n,\C)$. 
Classification of linear reversible maps is simple.
Suppose $F\in\GL(n,\C)$ is reversible.
Since the Jordan normal form of $F^{-1}$
consists of blocks of the same size as $F$
with inverse eigenvalues, 
the eigenvalues of $F$ that are not $\pm1$
must split into groups of pairs $\l,\l^{-1}$.
Furthermore, we must have the same number of Jordan blocks
of each size for $\l$ as for $\l^{-1}$.
Vice versa, if the eigenvalues of $F$ are either $\pm1$
or split into groups of pairs $\l,\l^{-1}$
with the same number of Jordan blocks of each size,
then both $F$ and $F^{-1}$ have the same Jordan normal form
and are therefore conjugate to each other.

\subsection{The Group $\GL(2,\C)$}
In particular, a linear map
is reversible in $\GL(2,\C)$
if and only if 
it is an involution or is conjugate to 
$\left(\begin{matrix}1&1\\0&1\end{matrix}\right)$,
$\left(\begin{matrix}-1&1\\0&-1\end{matrix}\right)$ 
or to
a matrix of the form
\beq\Label{E:glg}
 \left(\begin{matrix}\mu&0\\0&
\displaystyle\mu^{-1}\end{matrix}\right),
\eeq
for some $\mu\in\C^{\times}$.  
Thus each reversible $F$ is conjugate 
in $\5G$ (by a linear conjugacy) to
a map having one of these types as its linear part.

The collection of maps \eqref{E:glg}
forms an abelian subgroup, which we denote by $D$.
The element \eqref{E:glg} has infinite order
precisely when $\mu$ is not a root of unity,
and this is what we regard as the generic situation.
In this paper, we are going to concentrate on 
the following question:

\medskip
Q. Which elements $F\in\5G$ for which $L(F)$
has an eigenvalue that is not a root of
unity are reversible?

\medskip
To approach this, we are going to begin by
studying the centraliser of $D$ in $\5G$.
In fact, the classical Poincar\'e-Dulac
Theorem \cite[Section 4.8, Theorem 4.22]{PD}
implies that any map $F\in\5G_{n}$
is conjugate to a map in the centralizer of the linear part $L(F)$.
In case $L(F)={{\sf diag}}(\mu,\mu^{-1})\in D$ and $\mu$
is not a root of unity, it is easy to see that the centralizer of $L(F)$
coincides with the centralizer of $D$.
We shall, incidentally, discover some classes
of reversibles $F$ that have $L(F)=\id$, although our focus
is not on these non-generic examples.  Apart from the maps
given in \eqref{E:normal-0'}, which are tangent to
the identity when $\lambda=1$, we also meet the maps
given in Equation \eqref{E:normal-2} below.

\subsection{The Groups $\5C=C_D(\5G)$ and $\5E=E_D(\5G)$}\label{cent-section}
A map $F\in\5G$ is in the  {\em centralizer} of $D$
 if and only if it commutes with any particular
element of $D$ of infinite order, and if and only if
it takes the form $F=M(\phi,\psi)$, given by
\beq \Label{E:5C}
 M(\phi,\psi)(z) = (z_1\phi(z_1z_2),z_2\psi(z_1z_2)),
\eeq
where $\phi(t)$ and $\psi(t)$ are series in one
variable such that $\phi(0)\not=0\not=\psi(0)$ (i.e. 
they are series having nonzero constant term).  We denote by $\5C$ the group of all such maps:
$$\5C:=C_D(\5G)=\{M(\phi,\psi)(z) : \phi,\psi\in \5G_{1}, \phi(0)\not=0\not=\psi(0)\}.$$

In the following we shall
adopt the notation
$$p:=z_1z_2, \quad J(z)= \2z:=(z_2,z_1).$$
Then $J$ reverses every $\L\in D$, i.e.\ $J^{-1}\L J = \L^{-1}$.
Furthermore, a map $\T\in\5G$ reverses each (or any fixed infinite-order)
element $\L\in D$
if and only if $J\T$ commutes with $\L$ and hence, if and only if 
\beq \Label{E:5R}
\T(z) = (z_2\phi(p),z_1\psi(p)),
\eeq
where $\phi(t)$ and $\psi(t)$ are as before.
This may be written as
$\T=M(\phi,\psi)J=J M(\psi,\phi)$.
We denote the collection of such $\T$ by $\5R$:
$$\5R:=\{J F : F\in \5C \}=\{FJ : F\in \5C \}$$
and note that it is both a left and a right coset of $\5C$.

The {\em extended centraliser} $E_S(G)$ of a subset $S$ in a group
$G$ is the set of all elements of the group that
either commute with all the elements of the subset
or reverse them all.
We denote the extended centraliser of $D$
in $\5G$, by $\5E$:
$$\5E:=E_{D}(\5G)=\5C\cup\5R.$$
This is a group, in which $\5C$ has index $2$.

Lemma \ref{L:fol} applies to $\5H=\5C$ and to $\5H=\5E$:

\bl\Label{L:folE} If $F\in\5E$ has finite order, then there
exists $H\in\5C$ such that $F^H$ is linear.
\el

\subsection{The Homomorphisms $P$, $H$, and $\Phi$}\label{homs}
To $F\in\5E$ we associate the one variable power series 
$$P(F)=\rho(t)\in\5G_1, \quad \rho(t): = t\cdot\phi(t)\cdot\psi(t),$$
where $F$ is given by the right-hand side of either \eqref{E:5C} or \eqref{E:5R}.
Note that, following our convention, we use $\cdot$ for the (formal) pointwise multiplication 
of power series.
Denoting 
$$p=\pi(z_1,z_2):=z_1z_2,$$ 
we have the basic property
\begin{equation}\label{diagram}
P(F)\circ\pi = \pi\circ F,
\end{equation}
i.e.\ $P(F)$ is \lq\lq semi-conjugate" to $F$ via $\pi$.
Property \eqref{diagram} determines $P(F)$ uniquely.
A routine calculation using \eqref{diagram} proves:
\bl $P:\5E\to\5G_1$ is a group homomorphism.
\el

The map $P$ and its homomorphic property
have been fundamentally used in \cite{BZ} in
a more general context of one-resonant maps.

The kernel of $P$
is the set of maps $F$ of the form
$$ F(z) = 
\left(z_1\phi(p), \frac{z_2}{\phi(p)}\right)
\textup{ or }
\left(z_2\phi(p), \frac{z_1}{\phi(p)}\right),$$
and $\5C\cap\ker P$ is abelian.

Define the ``lifting'' homomorphism map $H:\5G_1\to\5C$ by 
\begin{equation}\Label{H}
H(\chi)(z) = \left( \frac{z_1\chi(p)}p, z_2\right),
\end{equation}
whose basic property is $P(H(\chi))=\chi$.
In particular, the restriction $P|\5C$
is surjective.
We also consider the ``lifting'' to the second argument given by
$$ (JH(\chi)J)(z) = \left(z_1,\frac{z_2\chi(p)}p\right)$$
that also has the basic property
$P(JH(\chi)J)=\chi$.

\medskip
We denote by $\5F_1^\times$ the multiplicative group
of the ring of formal power
series $\phi(t)=a_0+a_1t+\cdots$ in one variable
having $a_0\not=0$, 
(in which the group operation corresponds to convolution on the coefficients, the formal
equivalent of pointwise multiplication,
as opposed to composition). 
The map
$$\Phi:
\left\{
\begin{array}{rcl}
\5F_1^\times &\to& \5C,\\
\phi &\mapsto& 
\left(
z_1\phi(p),
\displaystyle \frac{z_2}{\phi(p)}
\right),
\end{array}
\right.
$$
is a group isomomorphism onto its image, which is  equal to $\5C\cap\ker P$.

\medskip
We note that each $F\in\5C$ has a unique factorization
in the form 
$H(\rho)\Phi(\psi)$, and another in the form
$JH(\rho)J\Phi(\phi)$. 
However, $\5C$ is not
the {\em direct} product of $\im H$ and $\im\Phi$.

\section{Centralisers in $\5E$}

\subsection{Centralisers in $\5C$}

A routine calculation gives:
\bl 
Let $F_j(z) = (z_1\phi_j(p),z_2\psi_j(p))$,
and $\rho_j=P(F_j)$, for $j=1,2$. Then
\beq\Label{E:product}
 F_1(F_2(z)) =
\left(
z_1\phi_2(p)\phi_1(\rho_2(p)),
z_2
\psi_2(p)\psi_1(\rho_2(p)) 
\right).
\eeq
\el
This immediately yields:
\bl\label{32}
 Let $F_j(z) = (z_1\phi_j(p),z_2\psi_j(p))$,
and $\rho_j=P(F_j)$, for $j=1,2$. Then
$F_1F_2=F_2F_1$ if and only if
\beq\Label{E:comm}
\left\{
\begin{array}{rcl}
\phi_2(p)\phi_1(\rho_2(p)) &=&
\phi_1(p)\phi_2(\rho_1(p))\\
\psi_2(p)\psi_1(\rho_2(p)) &=&
\psi_1(p)\psi_2(\rho_1(p))
\end{array}.
\right.
\eeq
\el

\bl\Label{L:key} Let $F_1\in\im\Phi$ $(=\5C\cap\ker P)$ and $F_2\in\5E$.  Suppose
$F_1F_2=F_2F_1$. Then either $F_1$ is linear (and hence is in the group $D$ of matrices \eqref{E:glg}) or $P(F_2)$ has finite order.
\el

The proof is based on the following useful fact in one variable:
\bl\label{onev}
Let $\phi(t)$ be a formal power series in one variable
that is invariant under a formal change of variables $t'=\rho(t)\in\5G_{1}$ with $\rho(0)=0$, i.e.\
\begin{equation}\label{inv}
\phi(\rho(t)) =\phi(t).
\end{equation}
Then either $\phi=\const$ or $\rho$ has finite order.
\el

\bpf
Suppose that $\phi\ne\const$, i.e.\
$$ \phi(t) = \phi(0) + \alpha t^k + \sum_{j=1}^\infty \alpha_j t^{k+j},\quad \a\ne0.$$
Since $\rho(t) = ct +\cdots$, comparing the coefficients
of $t^k$ in \eqref{inv} gives $c^k=1$. 
Then, replacing $\rho$ with $\rho^{k}$, we may assume $c=1$.
It clearly suffices to show that $\rho=\id$.
Assuming the contrary, we have
$$ \rho(t) = t(1 + \beta t^r +\HOT),\quad \b\ne0,$$
where 
$\HOT$ denotes terms of degree greater than $r$.  Now
$$\begin{array}{rcl}
 \phi(\rho(t)) &=& 
1+ \alpha\rho(t)^k +\sum_{j=1}^\infty \alpha_j \rho(t)^{k+j}
\\
&=& 1+ \alpha t^k(1+ k\beta t^r) + \sum_{j=1}^r \alpha_j t^{k+j}+\HOT
\end{array}
$$
where $\HOT$ denotes terms of degree greater than $k+r$, so comparing
coefficients of $t^{k+r}$ in \eqref{inv}, we get 
$\alpha\beta k=0$, a contradiction.
\epf

\bpf[Proof of Lemma~\ref{L:key}]
Replacing $F_{2}$ with $F_{2}^{2}$ if necessary,
we may assume that $F_{2}\in\5C$.
In the notation of Lemma~\ref{32} we have $\rho_{1}(t)=t$
and therefore the first identity in \eqref{E:comm}
implies $\phi_1(\rho_2(t)) =\phi_1(t)$.
Then by Lemma~\ref{onev},
either $\phi_{1}=\const$ and then $F_{1}\in D$ or $\rho_{2}=P(F_{2})$ has finite order.
\epf

\bc\Label{C:key} If commuting elements $F_1\in\im\Phi$
and $F_2\in\5C$ are tangent to the identity, then
$F_1=\id$ or $F_2\in \im\Phi$. In fact, it suffices to assume
that $F_{1}$ and $P(F_{2})$ are tangent to the identity.
\ec
\bpf By Lemma~\ref{L:key}, either $F_{1}$ is linear and hence $F_{1}=\id$
or $P(F_{2})$ has finite order and hence is the identity.
 \epf

\bc The centraliser of $\im\Phi$ in $\5E$ is $\im\Phi$.
\ec
\bpf Recall that $\5E=\5C\cup \5R$ (see Section~\ref{cent-section} for the notation).
Suppose $F\in\5E$ commutes with all elements of
$\im\Phi$. If $F\in\5R$, then it does not commute with,
for instance, 
$\left(\begin{matrix}2&0\\0&\frac12\end{matrix}\right)$,
so $F$ must belong to $\5C$, in particular, $\L:=L(F)\in D$.
Since $\Lambda$ commutes with each element of $\im\Phi$,
it follows that $F_2=\Lambda^{-1}F$ commutes with each element
of $\im\Phi$, and is tangent to the identity.
Taking $F_1$ to be any element of $\im\Phi$
that is tangent to the identity but not equal to the identity,
and applying the last corollary, we conclude that
$F_2\in \im\Phi$.
\epf 

\bc $\im\Phi$ is a maximal abelian subgroup of $\5E$.
\ec

\section{Reversibility in $\5E$}

Involutions are trivially reversible (by the identity), in any group. 
By Lemma \ref{L:folE},
each involution in $\5C$ is conjugate to
$(\pm z_1,\pm z_2)$,
and each proper involution in $\5R$ is conjugate to
$\pm(z_2,z_1)$.  It is not altogether
clear what other elements 
of $\5R$ are reversible in $\5E$, but we do not need
to know this, for our purposes.  We concentrate
on describing the elements of $\5C$ reversible in $\5E$.  

\subsection{Elements of $\im\Phi$}
\bl Each element of $\im\Phi$ is strongly-reversible 
in $\5E$.
\el
\bpf Such maps are obviously reversed by 
the involution $J$. 
\epf  


\subsection{Elements of $\5C$}
We remark that if $F\in\5C$ is reversible,
then $\det(L(F))=\pm1$, so $P(F)(t)=\pm t+\cdots$.
Thus $L(F)$ takes the form 
$\left(\begin{matrix}\pm\mu&0\\0&
\mu^{-1}\end{matrix}\right)$
for some nonzero $\mu\in\C$, so $L(F^2)\in D$.

\bt\Label{4.3}
 Let $F\in\5C$ be reversed by $\T\in\5E$, i.e.\
\begin{equation}\Label{reverse}
\T^{-1} F \T = F^{-1}.
\end{equation}
Then either
$P(F)$ has finite order or some power
of $\T$ belongs to the matrix group $D$. 
\et

We use the following simple relation between reversers and commuting maps:
\bl\label{commute}
If $F$ is reversed by $\Theta$, then $F$ is reversed by any $\Theta^m$ with $m\in\Z$ odd
and commutes with any $\Theta^m$ with $m\in\Z$ even.
\el

\bpf
The statement is straightforward for $m\ge 0$.
Taking the inverse of \eqref{reverse}
yields $\T^{-1} F^{-1}\T=F$ showing the statement for $m=-1$
and therefore for any $m<0$.
\epf

\bpf[Proof of Theorem~\ref{4.3}] Replacing $F$ by $F^2$, if need be,
we may assume $L(F)\in D$. 
Suppose $P(F)$ has infinite order.  Since $P(F)$
(which takes the form $t+\cdots$)
is reversed by $P(\T)$, it follows from Corollary~\ref{1-finite}
that $P(\T)$ has
finite order, so we may choose 
$k\in\N$ such that $\T^{2k}\in\im\Phi$.
Since $\Theta^{2k}$ also commutes with $F$
by Lemma~\ref{commute},
we may apply Lemma \ref{L:key} with
$F_1=\T^{2k}$ and $F_2=F$, and conclude
that $\T^{2k}\in D$. 
\epf


\bc\Label{C:FO}
 Suppose $F\in\5C$ has linear part $L(F)\in D$, with $L(F)\not=\pm\id$,
and is reversible in $\5E$.  Then (1) each reverser
of $F$ lies in $\5R$, and (2) $F$ may be reversed in $\5E$
by some element of finite order.
\ec
\bpf Suppose  $\T\in\5E$ reverses $F$. Then $L(\T)$
reverses $L(F)$, which
takes the form \eqref{E:glg},
with $\mu\not=\pm1$. Thus $L(\T)$  has to interchange
the eigenvectors of $\mu$ and $1/\mu$, and must
take the form
$(az_2,bz_1)$
for some nonzero $a$ and $b$. In particular,
the reverser cannot belong to $\5C$,
so part (1) is proved. 
Composing with an element
of $D$, we may assume that the reverser $\T$
has linear part $(cz_2,cz_1)$. Then Theorem~\ref{4.3}
tells us that $1^\circ$ $P(F)$ has finite order
or $2^\circ$ some power $\T^{2k}\in D$.

In case $1^\circ$, $P(F)$ is linearizable by Lemma \ref{L:fol}.
Furthermore, since $L(F)\in D$, the map $P(F)$ is also tangent to the identity.
Hence we must have $P(F)=\id$ and therefore $F\in \im \Phi$, 
which is reversed by $J$.

In case $2^\circ$, $\T^{2k}=L(\T)^{2k}=c^{2k}\cdot\id$, 
but the only multiples
of the identity belonging to $D$ are $\pm\id$,
so $\T$ has order at most $4k$.
\epf

\bc\Label{C:rev-FO} Suppose $F\in\5C$ has linear part $L(F)\in D$ with $L(F)\ne\pm\id$,
and is reversible in $\5E$. Then either $F$ is linear or
each reverser of $F$ in $\5E$ has finite order.
\ec

\bpf We have seen this in case $2^\circ$ (see Proof of Corollary~\ref{C:FO}).
In case $1^\circ$, $F\in \im \Phi$ is reversed by
$J$, so each other reverser of $F$ takes the form
$JG$, where $G$ commutes with $F$, and $G$
belongs to $\5C$ since $L(F)$ has two different eigenvalues. 
Then for $\L:=L(G)\in D$, the map $\L^{-1}G$ is tangent to the identity
and commutes with $L(F)^{-1}F$.
If $F\not=L(F)$, then Corollary \ref{C:key} tells us
that $\Lambda^{-1}G\in\im\Phi$, hence $G\in\im\Phi$,
so $(JG)^2=JGG^{-1}J=\id$.
\epf

We draw a further corollary from the proof  of Corollary \eqref{C:FO}:
\bc Suppose $F\in\5C$ is reversible in $\5E$ and has linear part $L(F)\ne\pm\id$.  Then $F$ is conjugate in $\5E$
to a map reversed by one of the maps
$J_c(z)=c\2z$, where $c$ is a root of unity.
\ec
\bpf In case $1^\circ$, we may take $c=1$.  In case
$2^\circ$, multiplying with suitable element $\L\in D$, 
$\T$ can be assumed to be of finite order and satisfy $L(\T)= J_c$, 
where $c$ is a root of unity. The statement now follows from Lemma~\ref{L:fol}.
\epf

\section{Maps Reversed by $J_c(z)=c\2z$}
Here we fix one of the linear maps $J_c\in\5R$ 
identified in the last subsection, and we describe
all the elements of $\5C$ that it reverses.
We assume that $c$ is a root of unity,
and set $\omega=\bar c^2$.
Let $k$ be the least natural number with $\omega^{2k}=1$
(i.e.\ $k$ is the order of $\omega^{2}$).

\subsection{} If a map $\Theta$ reverses two commuting maps 
$F$ and $G$, then it also reverses their composition $FG$.
Thus if $J_c$ reverses a map $F\in\5C$ having linear
part $\Lambda$, then, since $J_c$ reverses $\Lambda$,
it also reverses  the map $G=\Lambda^{-1}F$, which is tangent
to the identity. Thus each $F\in\5C$ reversed by
$J_c$ factors as $\Lambda G$, where $G\in\5C$ is tangent
to the identity, and is reversed by $J_c$.
 
\subsection{}\Label{5.2}
If $J_c$ reverses an $F\in\5E$. Then $\rho=P(F)$
is reversed by $c^2t=\bar\omega t$ (and hence by the inverse
$\omega t$, see Lemma~\ref{commute}), i.e.\
\beq\Label{reverse-1}
\omega^{-1}\rho(\omega\rho(t)) =t.
\eeq
In particular, Lemma~\ref{commute} implies
\begin{equation}\Label{o2}
\rho(\omega^2 t)=\omega^2\rho(t).
\end{equation}
Also reversibility implies $\rho(t)=\pm t+\HOT$ (see \S\ref{linear}).
Furthermore, since $J_c$ interchanges the eigenspaces of $L(F)$,
we must have $L(F)\in D$ and therefore 
\begin{equation}\label{rho}
\rho(t)=t+\HOT.
\end{equation}

\subsection{}\label{5.3}
Consider an arbitrary $F\in\5C$, with $L(F)\in D$, of the form 
$F=\left(z_1\phi(p), z_2\psi(p)\right)$, with
$\rho=P(F)$. 
For any complex $c$, we calculate
$$(F J_{c}^{-1} F J_c )(z)=\left( z_1\cdot\psi(c^2p)
\cdot\phi\left(p\cdot\phi(c^2p)\cdot\psi(c^2p)\right),
z_2\cdot\phi(c^2p)\cdot
\psi\left(p\cdot\phi(c^2p)\cdot\psi(c^2p)\right)
\right).$$
Thus $J_c$ reverses $F$ if and only if
\begin{equation}\Label{E:e-1}
\left.
\begin{array}{rcl}
\psi(c^2t)\cdot\phi\left(t\phi(c^2t)\psi(c^2t)\right)
&=& 1,\\
\phi(c^2t)\cdot\psi\left(t\phi(c^2t)\psi(c^2t)\right)
&=& 1,
\end{array}
\right\}
\end{equation}
or, equivalently,
\begin{equation}\Label{E:e-1'}
\left.
\begin{array}{rcl}
\psi(c^2t)\cdot\phi\left(\omega\rho(c^{2}t)\right)
&=& 1,\\
\phi(c^2t)\cdot\psi\left(\omega\rho(c^{2}t)\right)
&=& 1,
\end{array}
\right\}
\end{equation}
where as before $\rho(t)=t\phi(t)\psi(t)$.

Now define
\begin{equation}\Label{E:e-2}
\sigma(t):=\omega\rho(t),
\end{equation}
which satisfies
\begin{equation}\label{om}
\sigma(t)=\omega t+\HOT
\end{equation}
in view of \eqref{rho}.
Then the reversibility equation \eqref{reverse-1} for $\rho$ becomes
\begin{equation}\Label{E:e-4}
\sigma^2(t) = \omega^2 t,
\end{equation}
and the reversibility equations \eqref{E:e-1'} for $F$ become
\begin{equation}
\left.
\begin{array}{rcl}\Label{E:e-3}
\psi(t)\cdot\phi\left(\sigma(t)\right)
&=& 1,\\
\phi(t)\cdot\psi\left(\sigma(t)\right)
&=& 1,
\end{array}
\right\}
\end{equation} 
where we replaced $c^{2}t$ by $t$.
Denoting
\begin{equation}\Label{E:e-5}
 g(t): = \frac{\phi(t)}{\psi(t)},
\end{equation}
we obtain from \eqref{E:e-3} that
\begin{equation}\Label{E:e-6}
g(\sigma(t))=\frac{\phi(\sigma(t))}{\psi(\sigma(t))}
=\frac{\phi(t)}{\psi(t)}
=g(t).
\end{equation}
In view of \eqref{E:e-4}, it follows that 
\beq\Label{g-omega}
g(\omega^2t)=g(\sigma^2(t))=g(t).
\eeq

\medskip
Equation \eqref{E:e-4} admits two possibilities,
{\it a priori}:


\medskip
$1^\circ$:
$\sigma$ may be the linear map
$\sigma(t)=\omega t$. In this case, Equation \eqref{E:e-2}
yields $\psi(t)=1/\phi(t)$, so $F\in\im\Phi$.
Equations \eqref{E:e-3} then yield $\phi(\omega t)=\phi(t)$,
i.e. $\phi(t)$ takes the form $\phi_1(t^k)$,
where $k$ is the order of $\omega$.

Conversely, $J_c$ reverses $F(z)=(z_1\phi(p),z_2/\phi(p))$
whenever $\phi(c^2 t)=\phi(t)$, i.e. 
$\phi(t)$ is a function of $t^k$ and the order of
$c$ divides $2k$.

We note that each of these $F$ is reversed by
the involution $J=J_1$.

\medskip
$2^\circ$: The more interesting possibility is that $\sigma$ is a nonlinear
solution of \eqref{E:e-4}.  
Since $\omega$ is a root of unity,
$\sigma$ has finite even order.  Then Lemma~\ref{L:fol} implies the existence of some
 $h(t)=t+\cdots$ such that $\sigma^h(t)=\omega t$,
where $\sigma^h$ denotes the conjugate $h^{-1}\sigma h$, i.e.\ $h^{-1}(\sigma(h(t)))=\omega t$ or, equivalently,
\begin{equation}\label{ho}
h(\omega t)= \sigma(h(t)).
\end{equation}
Since $\sigma(\omega^2t)=\omega^2\sigma(t)$, it follows 
from the formula \eqref{inverse} that $h(\omega^2t)=\omega^2h(t)$.
Furthermore, setting $g_1(t)=g(h(t))$ we obtain using \eqref{ho} and \eqref{E:e-6}:
\begin{equation}\label{g1}
g_1(\omega t)= g(h(\omega t)) = g( \sigma(h(t))) = g(h(t)) = g_1(t).
\end{equation}
The equations
$$
\left.
\begin{array}{rcl}
\phi(t)\psi(t)
&=&
\displaystyle\frac{\rho(t)}t,\\
\displaystyle\frac{\phi(t)}{\psi(t)}
&=&
g(t),
\end{array}
\right\}
$$
are clearly equivalent to 
\beq\Label{E:e-9}
\begin{array}{rcl}
\phi(t) &=&
\left(\displaystyle
\frac{\rho(t) g(t)}t
\right)^{\frac12},
\\
\psi(t) &=&
\left(\displaystyle
\frac{\rho(t)}{tg(t)}
\right)^{\frac12},
\end{array}
\eeq
where the branches of the square roots are chosen
to make 
\begin{equation}\label{phi0}
\phi(0)=\frac1{\psi(0)}=\lambda,
\end{equation}
the first eigenvalue of $F$.
It follows from \eqref{E:e-9}, \eqref{o2} and \eqref{g-omega} that 
\begin{equation}\label{phi-inv}
\phi(\omega^2t)=\phi(t), \quad\psi(\omega^2t)=\psi(t),
\end{equation}
so that these functions, also, depend only on $t^k$.

\medskip
Conversely, suppose $c$ is a $4k$-th root of unity,
for some $k$, and take any invertible $h\in\5G_1$ with $h(t)=t+\HOT$ and
$h(\omega^2 t)=\omega^2h(t)$, i.e.\  $h$ is any power series
$h(t)=t(1+\sum_{j\ge 1} h_j t^{kj})$ . Define 
\beq\Label{E:e-7}
\sigma(t)=h(\omega h^{-1}(t)), \quad \rho(t)=\omega^{-1}\sigma(t) = \omega^{-1}h(\omega h^{-1}(t)),
\eeq
in particular, $\sigma^{h}(t)=\omega t$.
Then 
clearly both $\sigma$ and $\rho$ commute with $\omega^{2}$, and
$$ \sigma^2(t) = h(\omega h^{-1}(h(\omega h^{-1}(t)))) = \omega^2 t,$$
i.e. \eqref{E:e-4} holds,
which is equivalent to $\rho$ being reversed by $\omega t$.
  Note that $\sigma$ has order dividing
$2k$.
Take any $\lambda\not=0$ and
$g_1(t) = \lambda^{2} + \HOT$ satisfying $g_1(\omega t)=g_1(t)$,
and define 
\beq\Label{E:e-8}
g(t)=g_1(h^{-1}(t)),
\eeq 
so \eqref{g1} holds.
Since $h$ and $g_{1}$ commute with $\omega t$, $g$ also does, i.e.\ $g(\omega^2t)=g(t)$. Furthermore, 
$$ g(\sigma(t)) = g_1((h^{-1}(h(\omega h^{-1}(t)))))=g(t).$$
Finally define
holomorphic germs $\phi$ and $\psi$ by
\eqref{E:e-9}.
Then
$$
\psi\left(t)\phi(\sigma(t)\right) = 
\left(\displaystyle
\frac{\omega^{-1}\sigma^2(t)}{\sigma(t)}
\right)^{\frac12} 
\left(\displaystyle
\frac{\omega^{-1}\sigma(t)}{t}
\right)^{\frac12} 
=1,
$$
so the first equation in \eqref{E:e-3} holds.
The second equation follows from the first, using \eqref{E:e-5}
and \eqref{E:e-6} (and may also be verified by direct calculation).
Thus $J_c$ reverses $F$.

\be
An explicit example of $F$ in case $2^{\circ}$, reversed by $J_i(z)=i\2z$, 
and corresponding to the choice
$$ h(t)= \frac{t}{1-t},\quad g_{1}(t) = g(t) =\l^{2}, $$
and hence
$$
\rho(t) = -\sigma(t) = \frac{t}{1+2t}, \quad
\phi(t)=\frac{\l}{(1+2t)^{1/2}}, \quad \psi(t)=\frac{\l^{-1}}{(1+2t)^{1/2}},$$ 
is given by
$$ F(z_1,z_2) =
\left( \displaystyle\frac{\lambda z_1}{(1+2z_1z_2)^{1/2}} ,
\frac{ \lambda^{-1} z_2}{(1+2z_1z_2)^{1/2}}  \right).$$
Alternatively choosing
$$
g_{1}(t)=\frac{4\l^{2}}{1+t}, 
$$
and hence
$$
g(t) = g_{1}(h^{-1}(t)) = \l^{2}\frac{(1+t)^{2}}{1+2t},
\quad
\phi(t)=\l\frac{1+t}{1+2t}, \quad \psi(t)=\l^{-1}\frac1{1+t},$$ 
we obtain an example of a rational map
$$ F(z_1,z_2) =
\left( \displaystyle\frac{\lambda (1+z_{1}z_{2}) z_1}{(1+2z_1z_2)} ,
\frac{ z_2}{\l(1+z_1z_2)}  \right).$$
These maps are not reversed by $J$ or any other involution.
Indeed, if an involution $T$ reverses $F$,
then $L(T)=\pm J$ and hence $P(T)(t)=t+\HOT$ and therefore $P(T)(t)=\id$
since $T$ is an involution. But then $P(T)$ cannot reverse $P(F)$.
\ee

\medskip
Assembling the cases, we have
identified all the series $F\in \5C$ reversed by a  
given $J_c$, and can state a theorem:

\bt\Label{T:twokinds}
(1) Let $c^{2k}=1$. Then each $F\in\im\Phi$ reversed by
$J_c$ takes the form 
$$\left(z_1\phi(p),z_2/\phi(p)\right),$$
where $\phi(c^2t)=\phi(t)$.\\
(2) Let $c^{4k}=1$, and $\omega=\bar c^2$. 
Then each $F\in \5C$ with $L(F)={\sf diag}(\l,\l^{-1})$ that is reversed
by $J_c$ and does not belong to $\im\Phi$ takes the form
$(z_1\phi(p),z_2\psi(p))$, where $\phi$ and $\psi$
are defined by the equations \eqref{E:e-9} and \eqref{phi0},
$\rho$ is defined by \eqref{E:e-7}
for some $h\in\5G_1$ tangent to the identity that satisfies
$h(\omega^2 t)=\omega^{2}h(t)$, and  $g$ is defined by
\eqref{E:e-8} for some $g_1(t)=\l^2+\HOT$
that satisfies $g_1(\omega t)=g_1(t)$. 
In particular, both $\phi$ and $\psi$ depend only on $t^{k}$.
Moreover all these
maps are reversed by $J_c$.
\et

In part (2), the map $F$ may also be written in the notation of \S\ref{homs} in the form
\beq
F(z) = 
JH(\rho)J\Phi(\phi) = H(\rho)\Phi(1/\psi),
\eeq
where $\phi(\omega^2t)=\phi(t)$ by \eqref{phi-inv},
and $\rho=P(F)$ is reversed by $t\mapsto\omega t$ (see \S\ref{5.2}),
so that both $\phi(t)$ and $\rho(t)/t$ depend
only on $t^k$.  However, 
such maps $F$ are not reversed by $J_c$ in general, unless $\phi$ 
is constructed as above from $h$ and some $g_1$.

\section{Reversibility in $\5G$}

\subsection{Resonances}
\bl\Label{L:rev-res}
 If $F\in\5G$ is reversible in $\5G$, 
and has an eigenvalue that is not a root
of unity, then
$F$ is conjugate to some element of $\5C$
having $L(F)\in D$ of infinite order.
Moreover, the conjugating map can be chosen to be tangent to the identity.
\el
\bpf By a linear conjugation, we
may convert $L(F)$ to the
form $\L=\left(\begin{matrix}
\lambda&0\\0&\lambda^{-1}
\end{matrix}\right)\in D$.
Since $\lambda$
is not a root of unity, the 
only resonance relations are of the form
$$\l=\l^{k+1} (1/\l)^k, \quad 1/\l = \l^k (1/\l)^{k+1},$$
so the Poincar\'e-Dulac
Theorem \cite[Section 4.8, Theorem 4.22]{PD} tells us that 
$F$ may be conjugated to the resonant form \eqref{E:5C}.
\epf

\bl\Label{dTheta0} 
Suppose that $F\in \5G$ has $L(F)\in D$,
with $L(F)\not=\pm\id$,
and
is reversible. Then each reverser $\Theta\in\5G$
has linear part  of the form $L(\T)(z_{1},z_{2})=(az_{2},bz_{1})$.
Also, it is possible to choose a reverser
with linear part $J_c$.
\el
\bpf
The first assertion follows from the fact that $L(\T)$ must interchange the eigenspaces of $L(F)$.
In view of Lemma~\ref{L:rev-res}, we may assume that $F\in\5C$.
In general, the composition of a reverser of $F$
and an element of the centraliser $C_F(\5G)$ 
is another reverser of $F$.
Since $D\le C_F$, we may compose $\Theta$
with an element of $D$
to convert its linear part to the form
$J_c$.
\epf

\subsection{Terminology}
It is usual to say that a map is in {\em resonant form}
if the homogeneous terms in its expansion
commute with the linear part. For maps $F$ that belong
to $\5G$ and have $L(F)\in D$ of infinite order,
this just means that they belong to $\5C$, and
is independent of the particular $F$. For maps
with $L(F)$ of finite order, resonance amounts to 
a less restrictive condition. Since we are concentrating on
the generic case, we shall use the term {\em $D$-resonant map}
to mean an element of $\5C$.  Similarly, we
shall refer to all elements of $\5R$ as
{\em $D$-inverse-resonant maps}. 

More generally,
we extend this terminology to maps that
may not be invertible:

\bd
A formal power series
map $G\colon(\C^n,0)\to (\C^n,0)$
is called {\em $D$-resonant} (resp.\ {\em $D$-inverse-resonant})
if $G\circ M=M\circ G$ (resp.
 $G\circ M=M^{-1}\circ G$), whenever $M\in D$.
\ed

\br
So a series $G\colon(\C^2,0)\to (\C^2,0)$
with $L(G)\in D$ of infinite order is $D$-resonant (resp.\ $D$-inverse-resonant) 
if and only if it is the sum of monomials $p^k az$
(resp.\ $p^k a\2z$), where $a$ is some diagonal matrix, $p=z_{1}z_{2}$
and $\2z=(z_{2},z_{1})$.
\er

We have the following obvious properties:

\bl\Label{res-comp}
Let $G_1,G_2$ be $D$-resonant and $H_1,H_2$ be $D$-inverse-resonant.
Then $G_1\circ G_2$ and $H_1\circ H_2$ are $D$-resonant,
whereas $G_1\circ H_1$ and $H_1\circ G_1$
are $D$-inverse-resonant.
\el

\subsection{Reversers of Resonant Maps}
\bp\Label{Theta-res}
Suppose that $F\in\5C$, with $\L=L(F)$
of infinite order, and $F$ is reversed
by some $\Theta\in\5G$. Then $\Theta\in\5R$.
\ep

\bpf
By Lemma \ref{dTheta0}, the linear part of $\Theta$ belongs to $\5R$.
Assume by induction that all terms of order less than $k$ 
in the expansion of $\Theta$ are $D$-inverse-resonant.
Identifying homogeneous components of order $k$ in the basic reversibility equation
\begin{equation}
F\circ\T\circ F=\T,
\end{equation}
we obtain an identity $\L L_k(\Theta)(\L z)=L_k(\Theta)(z)+\ldots$ (using the notation introduced in \S\ref{l-notation}),
where the dots contain expressions involving  only $L_m(\Theta)$ with $m<k$.
It then follows from the inductive assumption and Lemma~\ref{res-comp}
that these terms are $D$-inverse-resonant.
Hence $\L L_k(\Theta)(\L z)-L_k(\Theta)(z)$
is $D$-inverse-resonant, which is only possible when 
$L_k(\Theta)$ is $D$-inverse-resonant, as is readily seen by using the
fact that $\Lambda$ has infinite order.
The proof is complete.
\epf

Combining Lemmas \ref{L:rev-res} and \ref{Theta-res}, we have:
\bt\Label{T:reduction}
Let $F\in\5G$ and suppose $L(F)$ has an eigenvalue that
is not a root of unity.
Then $F$ is reversible in $\5G$ if and only if 
it is conjugate in $\5G$ to some $D$-resonant element
$G\in\5C$, having $L(G)\in D$,
that is reversed by some $D$-inverse-resonant
element $\T\in \5R$.
\et

\section{Conjugacy Classes of Reversibles}\Label{classes}
By Theorem \ref{T:reduction}, 
each generic reversible of $\5G$ is
conjugate in $\5G$ to a map of the form
$\L F$, where $\L\in D$ and $F\in\5C$
is tangent to the identity and is reversible
in $\5E$. By Corollary \ref{C:FO}, $\L F$, and hence
$F$, may be reversed
by some element of finite order in $\5R$,
and by a further conjugation (using an element
of $\5C$, which does not disturb the factorization
$\L F$),
we may arrange that $F$ is reversed by 
some linear map, which may be taken
to be a $J_c$,
for some root of unity $c$.
By Theorem \ref{T:twokinds}, $F$ is of one of two 
kinds.
Now we turn to the question of cataloging the 
conjugacy classes in $\5G$ of the
maps of these two kinds.
 
\subsection{}
First we consider their conjugacy classes
in $\5E$.

The key idea is to use the conjugacy 
actions induced by the homomorphisms 
$H:\5G_1\to\5C$ and $\Phi:\5F_1^\times\to\5C$
on the group $\5C$ (as introduced in \S\ref{homs}).

\medskip
Consider general $F=M(\phi,\psi)\in\5C$
reversible or not,
given by \eqref{E:5C}. Let $\rho=P(F)$.
For $\chi\in\5G_1$, letting $K_1=H(\chi)$
and $K_2=JH(\chi)J$, 
we calculate
\beq K_1^{-1}FK_1(z) = 
\left( \frac{z_1 \rho^\chi(p)}{p\psi(\chi(p))},z_2\psi(\chi(p))
\right),\eeq
and, similarly,
\beq\Label{E:conj-K_2}
 K_2^{-1}FK_2(z) = 
\left( z_1\phi(\chi(p)), \frac{z_2 \rho^\chi(p)}{p\phi(\chi(p))}
\right).
\eeq
Also, for any $\phi\in\5F_1^\times$, we calculate
\beq\Label{E:conj-Phi(phi_1)} 
\Phi(\phi_1)^{-1}F\Phi(\phi_1)(z)=
\left(
\frac{z_1\phi(p)\phi_1(p)}{\phi_1(\rho(p))},
\frac{z_2\psi(p)\phi_1(\rho(p))}{\phi_1(p)}
\right).
\eeq

\medskip
Now consider the two kinds of reversible $F\in\5C$,
tangent to, but not equal to, the identity,
reversed by $J_c$, as in Theorem \ref{T:twokinds}:

\medskip
$1^\circ$: $F\in\im\Phi$, so $\rho=\id$.
\\
We have 
$\phi(t)=1+\alpha t^k+\HOT$, for some $k\in\N$
and $\alpha\not=0$.
Then we may choose $\chi\in\5G_1$ so that 
$\phi(\chi(t))=1+t^k$, and then
\beq\Label{E:normal-1}
 K_2^{-1}FK_2(z) = 
\left( z_1(1+p^k), \frac{z_2}{(1+p^k)}
\right).\eeq

Thus $F$
is represented,
up to conjugacy in $\5C$, by one of the maps
\eqref{E:normal-1}, 
for some $k\in\N$.

\medskip
$2^\circ$: $\rho\not=\id$.
$F$ takes the form
\beq\Label{E:normal-3} \left(z_1\phi(p),\frac{z_2\rho(p)}{p\phi(p)}\right) 
= JH(\rho)J\Phi(\phi),
 \eeq
for some reversible $\rho\in\5G_1$,
with $\rho=t+\HOT$, but $\rho\not=\id$,
(reversed by $\omega t$, where $\omega^k=-1$,
for some $k\in\N$)
and some $\phi(t)$ that depends only on $t^k$ (see \S\ref{5.3}).  
Then \cite{O} $\rho$ is conjugate to 
\beq\Label{1D-normal} f_k(t) = \frac{t}{(1-k t^k)^{\frac1k}}=
t+t^ {k+1}+\left(\frac{k+1}2\right)t^{2k+1}+\HOT,
\eeq
which is also reversed by all odd powers of $\omega t$
(where $\omega=\bar c^2$, as before).
Let $\chi\in\5G_1$ conjugate $\rho$ to $f_k$. Then
conjugating $F$ by $JH(\chi)J$, we may assume that
$\rho=f_k$.  Then the formal iterate $\rho^\alpha$
is given by 
\beq\Label{E:1D-flow} \rho^\alpha(t)= f_k^\alpha(t) = \frac{t}{(1-k\alpha t^k)^{\frac1k}}
= t\left(1+\alpha t^k+\HOT\right),
\eeq
whenever $\alpha\in\C$.

Choose $\alpha\in\C$ such that $\phi(t)=1+\alpha t^k+\HOT$
(note that $F$ is tangent to the identity).
Then $\frac{\rho^\alpha(t)}{\phi(t)}=t(1+\HOT)$, where
the higher terms involve only monomials $b_jt^{jk}$
with $j\ge2$, so we may choose $\phi_1(t)$, depending only on
$t^k$, such that
$$ \frac{\phi_1(t)}{\phi_1(\rho(t))}
= \frac{\rho^\alpha(t)}{t\phi(t)}.$$
The latter fact follows by writing $\phi_{1}(t)=1+\sum a_{j}t^{jk}$, expanding the formula
\begin{equation}\Label{arg}
1+\sum a_{j}t^{jk}=\Big(1+\sum_{j\ge 2} b_{j}t^{jk}\Big) \Big(1+\sum a_{j}(t(1+t^{k}+\HOT))^{jk}\Big),
\end{equation}
identifying coefficients of $t^{jk}$ and solving inductively for $a_{j-1}$.

Then the conjugation 
\eqref{E:conj-Phi(phi_1)} converts $F$ to the form
\beq\Label{E:normal-2} 
F = \left(
\frac{z_1\rho^\alpha(p)}{p},
\frac{z_2\rho(p)}{\rho^\alpha(p)}
\right),
\eeq
and this is reversed by
$$ \Phi(\phi_1)^{-1}J_c\Phi(\phi_1)= J_c\Phi(\phi_1^2).$$

We calculate (using the fact that $z\mapsto\omega z$ reverses $\rho^\alpha$)
that
$$
\begin{array}{rcl}
 F^{J_c}(z) 
&=& 
J_{\bar c}F(cz_2,cz_1)\\
&=& 
J_{\bar c}
\left(
\dsty\frac{cz_2\rho^\alpha(\bar\omega p)}{\bar\omega p},
\dsty\frac{cz_1\rho(\bar\omega p)}{\rho^\alpha(\bar\omega p)}
\right)\\
&=& 
\left(
\dsty\frac{cz_2\bar\omega\rho^{-\alpha}(p)}{\bar\omega p},
\dsty\frac{cz_1\bar\omega\rho^{-1}(p)}{\bar\omega\rho^{-\alpha}(p)}
\right)\\
&=& 
\left(
\dsty\frac{z_1\rho^{-1}(p)}{\rho^{-\alpha}(p)},
\dsty\frac{z_2\rho^{-\alpha}(p)}{p}
\right),
\end{array}
$$
$$
\begin{array}{rcl}
H(\rho)\Phi\left(\dsty\frac{\rho^{\alpha}(t)}{\rho(t)}\right)(z)
&=& H(\rho)
\left(
\dsty\frac{z_1\rho^{\alpha}(p)}{\rho(p)},
\dsty\frac{z_2\rho(p)}{\rho^\alpha(p)}
\right)\\
&=& 
\left(
\dsty\frac{z_1\rho^\alpha(p)}{\rho(p)}\frac{\rho(p)}p,
\dsty\frac{z_2\rho(p)}{\rho^\alpha(p)}
\right) = F(z),
\end{array}
$$
$$ 
\begin{array}{rcl}
F^{-1}(z) 
&=&
\Phi\left(\dsty\frac{\rho^\alpha(t)}{\rho(t)}\right)^{-1}H(\rho)^{-1}(z)\\
&=&
\Phi\left(\dsty\frac{\rho(t)}{\rho^\alpha(t)}\right)H(\rho^{-1})(z)\\
&=&
\Phi\left(\dsty\frac{\rho(t)}{\rho^\alpha(t)}\right)
\left(
\dsty\frac{z_1\rho^{-1}(p)}p,z_2
\right)\\
&=&
\left(
\dsty\frac{
z_1\rho^{-1}(p)}p 
\cdot\frac{\rho(\rho^{-1}(p)}{\rho^{\alpha}(\rho^{-1}(p))},
\dsty\frac{z_2 \rho^{\alpha}(\rho^{-1}(p))}{\rho(\rho^{-1}(p))}
\right)\\
&=& \left(
\dsty\frac{z_1\rho^{-1}(p)}{\rho^{\alpha-1}(p)},
\dsty\frac{z_2\rho^{\alpha-1}(p)}{p}
\right).
\end{array}
$$
But conjugation of $M(\phi,\psi)$ by $\Phi(\phi_1)$,
and hence
by $\Phi(\phi_1^2)$, does not change
the coefficient of $t^k$ in $\phi$, so comparing this
coefficient in the maps $F^{J_c}$ and $F^{-1}$,
we obtain $\alpha-1=-\alpha$, or $\alpha=\frac12$.
Thus $F$ takes the form
\beq\Label{E:normal-4} 
F(z) = \left(
\frac{z_1\rho^{\frac12}(p)}{p},
\frac{z_2\rho(p)}{\rho^{\frac12}(p)}
\right).
\eeq
Conjugating with $\nu(z)=\2z$ we obtain
\beq\Label{E:normal-4'} 
F(z) = \left(
\frac{z_1\rho(p)}{\rho^{\frac12}(p)},
\frac{z_2\rho^{\frac12}(p)}{p}
\right)
=
  \left(
\frac{(1-\frac{k}{2}p^k)^{\frac1k}}{(1-kp^{k})^{\frac1k}} z_{1},
\frac{1}{(1-\frac{k}{2} p^k)^{\frac1k}} z_{2}
\right).
\eeq
Conjugating further by suitable scaling, $\2F$ takes the form
 \beq\Label{norm-2}
 F(z)=
  \left(
\frac{(1+ p^k)^{\frac1k}}{(1+2p^{k})^{\frac1k}} z_{1},
\frac{1}{(1+ p^k)^{\frac1k}} z_{2}
\right).
 \eeq
The alternative forms
\beq\Label{E:normal-6} 
 \left(
\frac{(1-\frac{1}{2}p^k)}{(1-kp^{k})^{\frac1k}} z_{1},
\frac{1}{(1-\frac{1}{2} p^k)} z_{2}
\right),
\quad
 \left(
\frac{1}{(1-kp^{k})^{\frac1k} (1+\frac{1}{2}p^k)} z_{1},
\Big(1+\frac{1}{2} p^k\Big) z_{2}
\right)
\eeq
may be obtained from \eqref{E:normal-4'} by a conjugation of the type
\eqref{E:conj-Phi(phi_1)} with $\phi_{1}$ satisfying respectively
$$\frac{\phi_{1}(t)}{\phi_{1}(\rho(t))}=\frac{(1-\frac{1}{2}p^k)}{(1-\frac{k}{2}p^k)^{1/k}}
\quad \text{ or }\quad
\frac{\phi_{1}(t)}{\phi_{1}(\rho(t))}
=\frac{1}{(1-\frac{k}{2}p^k)^{1/k}(1+\frac{1}{2}p^k)}
.$$
The latter fact follows by the argument analogous to the one preceeding
\eqref{arg}.

\medskip
Vice versa, we have the following lemma that can be verified by direct calculation:
\bl\Label{reversed}
Let $\rho(t)=t+\HOT$ be reversed by the rotation $t\mapsto\omega t$
with $\omega=c^{-2}$. Then 
\beq\Label{E:normal-41} 
F(z) = \left(
\frac{z_1\rho(p)}{\rho^{\frac12}(p)},
\frac{z_2\rho^{\frac12}(p)}{p}
\right), \quad p=z_{1}z_{2},
\eeq
is reversed by $J_{c}$.
\el

\bpf[Proof of Theorem~\ref{main0}]
Summarizing, we obtain that any reversible map is formally conjugate
either to a linear map or to a map \eqref{E:normal-1} or
to a map \eqref{norm-2}, which proves the first assertion of Theorem~\ref{main0}.

To show that these map are pairwise inequivalent under conjugation,
note that the maps $F$ in \eqref{E:normal-1} have $P(F)=\id$, whereas the ones in \eqref{norm-2} have $P(F)$ conjugate to $f_{k}$.
As consequence of Poincar\'e-Dulac,
any conjugation map between those maps must be $D$-resonant,
i.e.\ in the centralizer $\5C$.
Consequently, the corresponding one-variable maps $P(F)$
must be conjugate.
This shows that the maps in \eqref{norm-2} are pairwise inequivalent
under conjugation and are not conjugate to any map in \eqref{E:normal-1}.
To see that also the maps in \eqref{E:normal-1} are pairwise inequivalent,
since any map in $\5C$ splits as $H(\chi)\Phi(\phi)$,
it suffices to observe any conjugation by $\Phi(\phi)$ is trivial,
whereas a conjugation by $H(\chi)$ is given by \eqref{E:conj-K_2}
and hence cannot change the integer $k$.

The statement (1) of Theorem~\ref{main0}
is evident for maps in \eqref{E:normal-1}
and follows from Lemma~\ref{reversed} for maps in \eqref{norm-2}.

To show the statement (2), observe that any map in \eqref{E:normal-1}
is reversed by the involution $J(z)=\2z$ and hence is strongly reversible.
On the other hand, suppose a map $F$ in \eqref{norm-2}
is reversed by an involution $\Theta$.
Then we know from Proposition~\ref{Theta-res}
that $\Theta$ must be $D$-inverse-resonant,
i.e.\ it reverses the linear part of $F$.
Since $\Theta$ is an involution
and $\l\ne\l^{-1}$, we must have $L(\Theta)=\pm J$
and therefore $P(\Theta)(t)=t +\HOT$.
But then, since $P(\Theta)$ is also an involution,
it follows that $P(\Theta)=\id$.
Consequently $P(\Theta)$ cannot reverse
$P(F)\ne P(F)^{-1}$ and hence $\Theta$ cannot reverse $F$.

Finally the statement (3) follows from Corollary~\ref{C:rev-FO}.
\epf

\bpf[Proof of Proposition~\ref{polynomial}]
In the foregoing argument, the specific form $f_k$ may be
replaced by any element $\rho(t)\in\5G_1$ 
that only has powers $t^{1+jk}$ and takes the form
$$ \rho(t) = t\left(
1+t^k+\left(\dsty\frac{k+1}2\right)t^{2k}+\HOT
\right),
$$
and we still obtain the normal form \eqref{E:normal-4} with the new $\rho$. 
To see this, note first that each such $\rho$
is conjugate in $\5G_1$ to $f_k$, and that the conjugating
map, say $\chi$, only has powers $t^{1+jk}$ and may be chosen equal to
be equal to the identity up to the terms of order $2k+1$.
It follows that $J_c$ commutes with $H(\chi)$ up to order $4k+1$
and therefore it reverses $F^{H(\chi)}$ up to terms 
of degree $4k+1$ in $z$, and, arguing as before, we can conjugate
$F^{H(\chi)}$ to the form \eqref{E:normal-2}, 
and then we still get $\alpha=\frac12$.
Hence $F^{H(\chi)}$ is conjugate to \eqref{E:normal-6}. 
Now, noting that
$$
\frac{\rho(t)}{t(1+\frac12 t^k)}
=
1+\frac12 t^k +\left(\frac{2k+1}4\right)t^{2k}+\HOT,$$
we may choose
$$\rho(t)=  t\Big(1+\frac12 t^k+\Big(\frac{2k+1}4\Big)t^{2k} \Big) 
\Big(1+\frac12 t^k\Big)$$
and get the form
$$ F(z) = 
\left(
z_1\Big(1+\frac12 p^k\Big),
z_2\left(1+\frac12 p^k+\left(\frac{2k+1}4\right)p^{2k}\right)
\right).$$
Conjugating by $z\mapsto\alpha z$, with $\alpha^{2k}=2$, we get
the tidier polynomial form
\beq\Label{E:normal-7}
F(z) = 
\left(
z_1(1+ p^k),
z_2\left(1+ p^k+(2k+1)p^{2k}\right)
\right).
\eeq
\epf

The alternate forms \eqref{E:normal-4}/\eqref{E:normal-7} each have their
advantages.  The second has the simpler form, but the first 
is reversed by the simple map $J_c$.

\br We have seen that
maps of the form \eqref{E:normal-2}
are only reversed by some $J_c$ when $\alpha=\frac12$. 
However, it is worth remarking that for every $\alpha\in\C$
they are reversed by the {\em $D$-resonant} maps $c\cdot H(\rho^\alpha)$,
as is readily checked.
This does not mean, of course,
that the original map is $\L F$ is reversible for $\a\ne\frac12$
but might be of interest
for the study of reversible $F\in\5G$
that are tangent to the identity.
\er

\section{Convergence properties of the normal form}

The convergence of the normal form as stated in Theorem~\ref{convergence}
is obtained by closely following the arguments of 
the proof of Theorem~4.1 in \cite{MW}, page 283.
In fact, assume that $F$ is biholomorphic and admits 
a formal Poincar\'e-Dulac normal form
\begin{equation}\Label{pd-normal}
\xi = M\xi, \quad \eta=M^{-1}\eta,
\end{equation}
(written in the form similar to (3.6) in \cite{MW}),
where $M$ is a formal power series in the product $\xi\eta$.
Changing the notation $(x,y)=(z_{1},z_{2})$ as  in \cite{MW} and
writing $F$ as
\begin{equation}
x'=\l x + f(x,y), \quad y'=\l^{-1} y + g(x,y)
\end{equation}
with $f$ and $g$ convergent of order at least $2$,
analogously to (3.2) in \cite{MW},
and the conjugation map $\psi$ into the normal form as
\begin{equation}
x=U(\xi,\eta)= \xi + u(\xi,\eta), \quad y=V(\xi,\eta)=\eta + v(\xi,\eta)
\end{equation}
with $u$ and $v$ of order at least $2$,
analogously to (3.3) in \cite{MW}, 
the fact that $\psi$ conjugates $F$ to \eqref{pd-normal}
can be written as
\begin{equation}\Label{UV}
U(M\xi,M^{-1}\eta)-\l U(\xi,\eta)=f(U,V), \quad
V(M\xi,M^{-1}\eta)-\l V(\xi,\eta)=g(U,V), 
\end{equation}
analogously to (4.1) in \cite{MW}.

As in \cite{MW} a formal power series $p(\xi,\eta)$
is said to have type $s$ if it can be written
$$p(\xi,\eta)=\sum_{i-j=s} p_{ij}\xi^{i}\eta^{j}$$
and any power series $p$ admits an unique decomposition
$$p(\xi,\eta)=\sum p_{s}(\xi,\eta)$$
with $p_{s}$ having type $s$.
According to the proof of Poincar\'e-Dulac,
we may assume that the conjugating map $\psi$
has no resonant terms, which as in \cite[(3.4)]{MW} can be written
using the type decomposition as
\begin{equation}\Label{normalization}
u_{1}=0, \quad v_{-1}=0.
\end{equation}

As in \cite[(4.2)]{MW}, taking terms of the same type in \eqref{UV}
yields
\begin{equation}
(M^{s}-\l)U_{s}=[f(U,V)]_{s}, \quad (M^{s}-\l^{-1})V_{s}=[g(U,V)]_{s}.
\end{equation}
Then proceeding as in \cite{MW}, pages 284--286,
we obtain a convergent series $W(\xi)$
majorizing $u(\xi,\xi)$ and $v(\xi,\xi)$
and hence proving the convergence of
the conjugating map $\psi$.
Note that this proof gives the convergence
of the (unique) map $\psi$
satisfying \eqref{normalization}
and conjugating $F$ to a normal form \eqref{pd-normal}.

Finally, once $F$ is biholomorphically conjugate to \eqref{pd-normal}
with $M$ convergent, its further biholomorphic conjugation
into a form \eqref{E:normal-1} is obtained
by suitable $K_{2}=H(\chi)$ associated with convergent $\chi$.
This ends the proof of Theorem~\ref{convergence}.

\medskip
The following example shows that,
in contrast to the first series in \eqref{E:normal-0'},
biholomorphic reversible maps 
conjugate to a map in the second series in \eqref{E:normal-0'}
may fail to be biholomorphically conjugate to their formal normal form.

\be\Label{divergent}
Let $\rho(t)=t+\HOT$ be any one-variable
biholomorphic map that is formally but not biholomorphically reversible,
see \cite[Example~4.8]{AO} for the existence of such maps.
Let $\theta$ be a formal map reversing $\rho$.
Then any two-variable map 
$$F:=\L H(\rho) = \left(
\frac{\rho(p)}{p} \l z_{1}, \l^{-1}z_{2}
\right) $$
is formally reversible by $H(\theta)$.
On the other hand, any biholomorphic reverser $\Theta$ of $F$
with $\l$ not a root of unity,
would be inverse-resonant by Proposition~\ref{Theta-res}.
But then $\rho=P(F)$ would be reversible by the biholomorphic
map $P(\Theta)$, which is impossible due to the choice of $\rho$.
Hence $F$ cannot be biholomorphically reversible.
In particular, it cannot be biholomorhically conjugate
to any map in \eqref{E:normal-0'}.
\ee

\section{Factorization in $\5G$}\Label{factor}
The homomorphisms $\Phi$ and $H$ may also be applied 
to resolve another question.
It is an interesting fact that in many very large
groups each element may be factored as the product of 
a fixed small number of elements of a handful of conjugacy classes.
For instance, any permutation of a finite set is a product of transpositions,
and also the product of two involutions.
Of particular interest are products of involutions,
and, more generally, products of reversibles.
In the present case, we have the following:

\bt\Label{T:prod-4-rev}
If $F\in\5G$ has $\det L(F)=1$, then
it may be factorized as $F=
g_1 g_2 g_3 g_4$, where each $g_j$
is reversible in $\5G$.
\et

\br Each product 
$F=f_1\cdots f_n$
of reversible $f_j$'s has 
$\det L(F)=\pm1$, so (multiplying if
necessary by a suitable linear involution)
it follows from the theorem that each product of
reversibles reduces to the product of
five.  It also follows that the elements that
are products of reversibles are precisely those
with $\det L(F)=\pm1$.
\er

\bpf[Proof of Theorem~\ref{T:prod-4-rev}]
In fact, if $\det L(F) =1$, then 
{conjugating to 
a Jordan normal form  and
multiplying by some (reversible) $\Lambda\in D$
we can arrange that $L(F\2\Lambda)$ is conjugate to an
infinite-order element of $D$, where $\2\L$ is conjugate to $\L$
and therefore reversible. }
Then by Poincar\'e-Dulac, $F\2\Lambda$
is conjugate (say by $K\in\5G$) 
to some element of the centralizer $\5C$, so that $(F\2\Lambda)^K$ is resonant and hence may be
factored as $H(\chi) \Phi(\phi)$, where
$\chi(t)=t+ \HOT$. Now $\Phi(\phi)$
is reversible, 
and we know \cite[Theorem 9(2)]{O} 
that $\chi$ is the product
of two reversibles in $\5G_1$, so $H(\chi)$ is the
product of two reversibles, say $H(\chi_1)$ and $H(\chi_2)$.
Thus  
$$F^K= H(\chi_1) H(\chi_2) \Phi(\phi) 
(\2\Lambda^{-1})^K$$
is the product of four reversibles, and conjugating with $K^{-1}$ we obtain the result.
\epf	

Theorem~\ref{prod} now follows from Theorem~\ref{T:prod-4-rev}.
Indeed, if $\det L(F)=1$, it is the product of $4$ reversibles
with the involution equal to the identity.
Otherwise if $\det L(F)=-1$, consider the involution $\nu(z_{1},z_{2})=(-z_{1},z_{2})$.
Then $\det L(F\nu)=1$ and hence $F\nu$ is the product of $4$ reversibles,
implying the result.

 
\section{Reversible Biholomorphic Maps}
For series in one variable, the single formal conjugacy
class of $f_k$ in $\5G_1$, intersected with
the subgroup $\6G_1$ of biholomorphic maps splits into uncountably
many conjugacy classes.  Functional moduli
for these classes have been provided by \'Ecalle and Voronin\cite{PD}.
It is not necessarily true that every formally-reversible 
biholomorphic map is biholomorphically-reversible,
but because of the fact that all reversers are of finite order
it is true that every reversible
biholomorphic map is conjugate to one that is reversed by 
a rational rotation.  The same principle carries
over to our present context:

\bt\Label{biholo}
 Let $F\in\6G$ be an invertible biholomorphic germ
on $(\C^2,0)$, and suppose that $L(F)$ has an
eigenvalue that is not a root of unity.  Then $F$ is reversible
in $\6G$ if and only if it is conjugate in $\6G$
to a map that is reversed by linear map of finite order.
\et

\bpf Suppose $\T\in\6G$ reverses $F\ne\id$ in $\6G$. Then $\T$ reverses
$F$ in $\5G$ by Theorem~\ref{main0}, and hence has finite order. Thus, by Lemma \ref{L:fol},
$\T$ is conjugate in $\6G$ to a linear map. Applying the
same conjugation to $F$, we obtain the result.
\epf 

\br It remains open, even for one-variable maps, whether
results such as Theorems \ref{T:prod-4-rev}
or \ref{prod} hold for {biholomorphic maps}.
\er

\end{document}